\documentclass{birkjour}

\usepackage{amsmath, amssymb}
\usepackage[shortlabels]{enumitem}
\usepackage{mathrsfs}
\usepackage{amsthm}
\usepackage{geometry,tikz}
\newtheorem{theorem}{Theorem}
\newtheorem{prop}{Proposition}
\newtheorem{definition}{Definition}
\newtheorem{corollary}{Corollary}

\newtheorem{lem}{Lemma}
\newtheorem{atheo}{Theorem}

\everymath{\displaystyle}
\usepackage{xcolor}

\begin{document}
	\title[An IDE with a discontinuous kernel]{Asymptotic regimes of an integro-difference equation with discontinuous kernel} 
	
	\author{Omar Abdul Halim}
	\address{Department of Mathematics and Statistics\\
		University of Northern British Columbia\\
		Prince George, BC, V2N 4Z9}
	\email{abdul@unbc.ca}
	\author{Mohammad El Smaily}
	\address{Department of Mathematics and Statistics\\
		University of Northern British Columbia\\
		Prince George, BC, V2N 4Z9}
	\email{smaily@unbc.ca}
	\thanks{Mohammad El Smaily acknowledges partial support from the Natural Sciences and Engineering Research Council of Canada through the NSERC Discovery Grant RGPIN-2017-04313.}

	
	
	\date{October 14, 2021}
	
	\begin{abstract}This paper is concerned with an integral equation that models discrete time dynamics of a population in a patchy landscape. The patches  in the domain are reflected through the discontinuity of the kernel of the integral operator at a finite number of points in the whole domain. We prove the existence and uniqueness of a stationary state under certain assumptions on the principal eigenvalue of the linearized integral operator and the growth term as well. We also derive criteria  under which the population undergoes extinction (in which case the stationary solution is 0 everywhere).
	\end{abstract}
	\maketitle
	\section{Introduction and main results}\label{intro}
	In this paper, we study the long term dynamics of a population in a heterogeneous landscape. The density function of this population obeys an integral equation, which we describe in what follows. The population density in the  $n$th generation, at a location $x,$ is denoted by $u_n(x).$ The growth phase is described by some non-negative function $F,$ and the dispersal phase is described  by a dispersal kernel $k(x, y).$ The probability that an individual, who started its dispersal process at $x,$ will settle in $[y, y + dy)$ is then given by the product $k(x, y)dy.$   The population density in the next generation, at a location $x,$ is then obtained by summing up arrivals at $x$ from all possible locations $y.$ This yields the integral equation
	\begin{equation}\label{IDE}
		u_{n+1}(x)=\int_\Omega k(x,y)F(u_n(y)) \,{\rm d}y,
	\end{equation}
where $u_n(x)$ stands for the density of the population in the $n^\text{th}$ generation  at a position $x.$ 

Let us mention  some past works that relate to the above model. Musgrave and Lutscher \cite{Lutscher.patchy} considered \eqref{IDE} in the case where the domain $\Omega$ consists of patches that result a difference in the dispersal behaviour. At the interface of two patch types, the authors of \cite{Lutscher.patchy} incorporate recent results of Ovaskainen and Cornell \cite{Cornell}  that, in general, lead to a discontinuous density function for the random walker. In other words, a discontinuity in the dispersal kernel $k$ at the interface of these patches appears in the study done in \cite{Lutscher.patchy}. In \cite{Lutscher.patchy}, it is shown that the dispersal kernel can be characterized as the Green’s function of a second-order differential operator. Later, Watmough and Beykzadeh \cite{Watmough} made generalization of the classic Laplace kernel, which includes different dispersal rates in each patch as well as different degrees of bias at the patch boundaries. We also mention the work of Lewis et al. \cite{Lewis}, which considers the same model \eqref{IDE}, but with a different set of assumptions on the kernel $k.$ In \cite{Lewis}, the kernel $k$ is assumed to be  in the form $k(x,y)=k(x-y).$ The more important difference between our work and \cite{Lewis} is that the kernel $k$ is assumed to be continuous in \cite{Lewis}. In our present work, we allow discontinuity of $k$ at a finite number of points of the domain. For more details on the nature of discontinuity of our kernel, see the assumption \ref{discontinuous} below.  Roughly speaking, the discontinuity points account for a different dispersal behaviour; hence, a change in the patch where the population is moving. The discontinuity of $k$ adds more technicality  to the proofs of the main results than in \cite{Lewis}, especially those involving comparison arguments.

Our model \eqref{IDE} is a discrete time analogue of  the (continuous) time-space reaction-diffusion model 
	\begin{equation}\label{RD}
	\frac{\partial u}{\partial t}(t,x)=\frac{\partial^2 u}{\partial x^2}+f(u).
	\end{equation}
In \eqref{RD}, $t$ stands for the continuous time variable and $x$ stands for the continuous space variable. A higher dimensional version of \eqref{RD}, which also accounts for an underlying advection and heterogeneity in the landscape, is the semilinear parabolic partial differential equation
 \begin{equation}\label{RDA}
	\frac{\partial u}{\partial t}(t,x)=\Delta u+q(x)\cdot \nabla u+f(u), \quad t>0,\quad x\in \mathbb{R}^N,
	\end{equation}
for some incompressible vector field $q:\mathbb{R}^N\rightarrow \mathbb{R}^N.$
When \eqref{RDA} is used to describe the evolution of a population density $u(t,x),$ it is assumed that the dispersal of the species follows the normal distribution with zero mean, which however is not the case for most species. Equations of the type \eqref{RDA} have been studied in much detail. Traveling fronts, or pulsating traveling fronts, form a  particular class of solutions to such equations. The existence and qualitative properties of these travelling fronts and their propagation speeds have been studied in a long list of works (see \cite{ehom},   \cite{ek3}, \cite{ek1}, \cite{ek2} and the references therein). 

Equation \eqref{IDE} represents an alternative discrete-time model often used in biological literature. An advantage of the IDE \eqref{IDE} over the reaction-diffusion equation \eqref{RD} lies in the fact that \eqref{IDE} can be used to model the spatial spread of long-distance dispersers, while \eqref{RD} is inappropriate in this situation. 
	
	In this paper, we study the convergence of the integro-difference equation \eqref{IDE}, where $\Omega=(-a,a).$ We make the following assumptions:
	\begin{enumerate}[{(H1)}]
		\item\label{discontinuous} The function $k\ge0$ is a bounded nonnegative function on $\Omega\times \Omega,$ such that
		\begin{equation}\label{k.bounded}\text{for all }(x,y)\in\Omega^2, ~\delta< k(x,y)\leq \Lambda,\end{equation} for some $\delta,~\Lambda>0.$
		 Moreover, we assume that there exists a finite set  of points $\{a_i\}_{1\le i\le n}\subset\Omega,$ dividing the domain $\Omega$ into intervals \[\Omega_i=(a_i,a_{i+1})\text{ for  }1\le i\le n-1,~~\Omega_0=(-a,a_1)\text{ and }\Omega_n=(a_n,a)\]
		 and that $k$ is continuous on each $\Omega_i\times \Omega_j$ for $1\leq i,j\leq n.$
		\item\label{F_non_negative}  The nonlinearity $F$ satisfies: $F$ is continuous,  $0\leq F(x)\leq M$ for all $x\in \mathbb{R}.$ 
		\item\label{on.F} The function $F$ is strictly increasing on $[0,\infty)$ and it vanishes elsewhere. We assume that $F$ differentiable at $0$ and we set \begin{equation}\label{r0}
		r_0:=F^\prime (0)>1.
		\end{equation} 
		Furthermore, the function $F$ satisfies \begin{equation}\label{F.cond}\frac{F(u)}{u}<\frac{F(v)}{v},~  \text{ for all } u>v>0.\end{equation}
	\end{enumerate}
In this work, we will study the convergence of the sequence $u_{n+1}=\mathcal{T}(u_n),$ where $\mathcal{T}$ is defined by 
	\begin{equation}\label{define.T}\mathcal{T}(u)(x)=\int_\Omega k(x,y)F(u(y))dy,~ \text{ for }~u\in L^2 (\Omega).\end{equation}

	\begin{definition}[Stationary solution]
		A measurable function $w$ : $\Omega\to\mathbb{R}$ is called a stationary solution of \eqref{IDE} if it satisfies $w=\mathcal{T}(w)$.
	\end{definition}
	\begin{definition}\label{T0}
		We denote the linearization of  the nonlinear operator $\mathcal{T},$ at $u=0,$ by $$\mathcal{T}_0(u)(x):=r_0\int_\Omega k(x,y)u(y)dy, ~x\in\Omega.$$
We denote by  \begin{equation}\label{set.X}
		X=\left\{u\in L^2(\Omega):\,\text{$u$ is continuous on $\Omega$ except at finitely many points}\right\}.\end{equation}
	\end{definition}

	We will make use of Krein-Rutman theorem in analyzing the linearized operator $\mathcal{T}_0,$ introduced in Definition \ref{T0}. We recall the statement of this theorem, below. The  proof of this  result can be found in \cite{Deimling}.  
	\begin{atheo}[Krein-Rutman Theorem, \cite{Deimling} Theorem 19.2]\label{krein} Let $E$ be a Banach space and let $K\subset E$ be a closed convex cone such that $K-K$ is dense in $E.$ Let $\mathcal{A}:E\to E$ be a non-zero compact linear operator which is positive, meaning that $\mathcal{A}(K)\subset K,$ with positive spectral radius $r(\mathcal{A}).$ Then, there exists a principal eigenvalue $\lambda_0=r(\mathcal{A})>0$ of $\mathcal{A}$, and an eigenfunction $u\in K\setminus \{0\}$ of $\mathcal{A}$ such that $\mathcal{A}u=\lambda_0 u$.
		
	\end{atheo}
	
	\noindent We will apply  Theorem \ref{krein} to the operator $\mathcal{T}_0,$ where we let  the Banach space $E$ be $L^2(\Omega)$ and take $$K=L^2(\Omega)^+,$$ where \[L^2(\Omega)^+:=\{u\in L^2(\Omega), \text{ such that }u\ge0\text{ in }\Omega\}.\] 
	
	\begin{prop}\label{principal_eigen}  The operator $\mathcal{T}_0$ maps $L^2(\Omega)$ into itself. Moreover, $\mathcal{T}_0$ is a  positive  compact operator and there exists an eigenfunction $\phi_0\in K\setminus\{0\}$   of $\mathcal{T}_0$  that corresponds  to the principal eigenvalue $\lambda_0:=r(\mathcal{T}_0).$ 
	\end{prop}
	\begin{proof} We know that $K$ is a cone  and the set $K-K$ is dense in $L^2(\Omega).$ The map $\mathcal{T}_0$ maps  $L^2(\Omega)$ into itself because
		\[\begin{array}{ll}\|\mathcal{T}_0(u)\|_2^2&=r_0^2\int_\Omega \left(\int_\Omega k(x,y)u(y)dy\right)^2{\rm d}x\vspace{5 pt}\\
			&\leq r_0^2 \int_\Omega\|k(x,\cdot)\|_2^2\|u\|_2^2\,{\rm d}x\vspace{5 pt}\\&
			=r_0^2\|k\|_2^2\|u\|_2^2<\infty.
		\end{array}\]
		Since $k\in L^2(\Omega^2),$ it then follows (see Lax \cite{Lax}, for eg.)  that $\mathcal{T}_0$ is a  compact linear operator. Moreover,  if $u\in K$ and  $x\in \Omega$, then $$\mathcal{T}_0(u)(x)= r_0\int_\Omega k(x,y)u(y)\,dy\ge r_0\delta\int_{-a}^au(y)\,dy\ge0$$ as $u\ge0$ and $u\not\equiv 0$. Hence, $\mathcal{T}_0(u)\ge0$ and $\mathcal{T}_0(u)\in K $. Then, $\mathcal{T}_0$ is positive.
		
		Since $\mathcal{T}_0$ is a bounded linear operator, the spectral radius of $\mathcal{T}$ is (see \cite{Lax})
		
		\[\lim_{n\to\infty}\|\mathcal{T}_0^n\|^\frac{1}{n}=r(\mathcal{T}_0)\]
		For $v(x)=\frac{1}{\sqrt{|\Omega|}},$ we have $\|v\|_2=1.$ Using $v$, we obtain  
		\[\begin{array}{c}
		  \|\mathcal{T}_0^n\|=\sup_{\|u\|_2\le1}\| \mathcal{T}_0^n(u) \|_2 \ge \|\mathcal{T}_0^n(v) \|_2\vspace{3mm}\\
		   \ge \left(\int_\Omega \left(r_0\int_\Omega k(x,y_1)\mathcal{T}_0^{n-1}(v)(y_1){\rm d}y_1\right)^2{\rm d}x\right)^\frac{1}{2}\vspace{3mm}\\
		      =\left(\int_\Omega \left(r_0^n\int_\Omega k(x,y_1)\int_\Omega k(y_1,y_2)\int_\Omega\cdots\int_\Omega k(y_{n-1},y_n) v(y_n){\rm d}y_n\cdots{\rm d}y_1\right)^2{\rm d}x\right)^\frac{1}{2}\vspace{3mm}\\
		     \ge r_0^{n}\left(\delta^{2n}\int_\Omega \left(\int_\Omega \cdots\ \int_\Omega \frac{1}{\sqrt{|\Omega|}}{\rm d}y_n\cdots{\rm d}y_1\right)^2{\rm d}x\right)^\frac{1}{2}\vspace{3mm}\\
		     =r_0^n\delta^n\frac{|\Omega|}{|\Omega|}^{n+1}=r_0^n\delta^n|\Omega|^n.
		\end{array}\]
		Taking the $n$th root both sides and the limit as $n$ goes to $\infty$ we get that 
		\begin{equation}\label{lower.estimate}
		r(\mathcal{T}_0)\ge r_0\delta|\Omega|>0.
		\end{equation}
	This proves that the spectral radius is positive ( i.e. $r(\mathcal{T}_0)>0$). 
		We can then apply Theorem \ref{krein} to obtain the existence of an eigenfunction 
		$\phi_0\in K$   of $\mathcal{T}_0$ that corresponds  to the principal eigenvalue $\lambda_0:=r(\mathcal{T}_0).$ 
	\end{proof}

	Our first two results answer the question about the stationary state of \eqref{IDE} in the case where $F(0)=0.$ The main criterion used to decide the nature of the stationary solution is the value of the principal eigenvalue $\lambda_0.$ 
	\begin{theorem}\label{convergence_to_equil}
		Let $\{u_n\}_n$ be a solution of \eqref{IDE}, with the initial condition $u_0\in X$. Suppose that $F$ and $k$ satisfy the assumptions \ref{discontinuous},  \ref{F_non_negative} and \ref{on.F}. Assume furthermore that $F(0)=0.$ If $\lambda_0\le1$ , then
		
		\begin{enumerate}
			\item $0$ is the only stationary solution of \eqref{IDE} in $X$.
			\item  The sequence $\{u_n\}_n$ converges to $0$ in $L^2(\Omega),$ as $n\to\infty$.
			
		\end{enumerate}
	\end{theorem}
	
	Theorem \ref{convergence_to_equil} shows that the population will eventually undergo extinction, whenever the principal eigenvalue $\lambda_0$ is less or equal 1. The next theorem shows that the threshold $\lambda_0=1$ is sharp in the sense that  the population settles at a positive stationary state whenever $\lambda_0>1.$

	As a  consequence of Theorem \ref{convergence_to_equil}, we can now consider the case with mortality. In general, mortality of individuals is included in the dispersal \cite{Lutscher.patchy}. This can be reflected on the kernel by assuming that 
\begin{equation}\label{mortality}
    \int_\Omega k(x,y)\,\mathrm{d}y\le1 \text{ for all $x\in \Omega.$}
\end{equation}
Under the assumption \eqref{mortality}, the following result holds:
\begin{corollary}\label{r0Theorem}
	Let $\{u_n\}_n$ be a solution of \eqref{IDE}, with the initial condition $u_0\in X$. Suppose that $F$ and $k$ satisfy the assumptions \ref{discontinuous},  \ref{F_non_negative}, \ref{on.F} and the additional assumption \eqref{mortality}. Assume furthermore that $F(0)=0$ and $r_0\le1.$ Then,
		\begin{enumerate}
			\item $0$ is the only stationary solution of \eqref{IDE} in $X$.
			\item  The sequence $\{u_n\}_n$ converges to $0$ in $L^2(\Omega),$ as $n\to\infty$.
			\end{enumerate}
\end{corollary}
Our third result in the case where $F(0)=0$ is the following:	
	\begin{theorem}\label{Nonzero_Equilibrium}  Suppose that $F$ and $k$ satisfy the assumptions \ref{discontinuous}, \ref{F_non_negative} and  \ref{on.F}. Assume furthermore that $F(0)=0.$ Suppose that $u_0\in X,$  $u_0\geq0,$ $u_0\not\equiv 0$ and that $\lambda_0>1.$ Then,
		\begin{enumerate}
			\item There exists a unique positive stationary solution $w$ of \eqref{IDE}.
			\item The sequence $\{u_n\}_n$ converges to $w$ in $L^2(\Omega),$ as $n\to\infty$.
		\end{enumerate}
	\end{theorem}
	In the next theorem, we will see that the stationary state will be positive, regardless of $\lambda_0$, whenever $F(0)>0.$
		\begin{theorem}\label{case_F(0)_>0}
		Suppose that $F$ and $k$ satisfy the assumptions  \ref{discontinuous}, \ref{F_non_negative} and \ref{on.F}. Assume furthermore that $F(0)>0.$ If $u_0\in X$, $u_0\ge0$ and $u_0\not\equiv 0,$ then
		\begin{enumerate}		
			\item There exists a unique positive stationary solution $w$ of \eqref{IDE}. 
			\item  The sequence $\{u_n\}_n$ converges to $w$ in $L^2(\Omega),$ as $n\to\infty$.
		\end{enumerate}		 
	\end{theorem}

	\section{Proofs of the main results}
	We start with a series of Lemmas that will be used in the proofs of the theorems that we announced in Section \ref{intro}, above. 
	\begin{lem}\label{equilibrium.is.positive} Let $w$ be a stationary solution of \eqref{IDE}. Then, $w\ge0$. Moreover, if $w\not\equiv0$, then $w(x)>0$ for all $x\in\Omega$.
	\end{lem}
	\begin{proof}First we mention that, since $F$ is nonnegative and $k>0,$ then the function $w$ is nonnegative.  We prove the claim according to the value of $F(0).$ If $F(0)>0,$ then \[\forall x\in \Omega,\quad w(x)=\int_\Omega k(x,y)F(w(y))\,{\rm d}y>2a\delta F(0)>0.\]
If $F(0)=0,$ then for $x\in\Omega,$ using \eqref{k.bounded} we have 
\[w(x)=\int_\Omega k(x,y)F(w(y))\,{\rm d}y\geq \delta \int_\Omega F(w(y))\,{\rm d}y>0,\]
because if $\int_\Omega F(w(y))\,{\rm d}y=0,$ then $F(w(y))=0$ a.e in $\Omega.$  The latter implies that $w(y)=0$ a.e in $\Omega.$ However, as $w\not\equiv0$, there exists $z\in \Omega$ such that $w(z)>0.$ Then, $$0<w(z)=\int_\Omega k(z,y)F(w(y))\,dy=0$$ as $w(y)=0$ a.e in $\Omega$. This is a contradiction.
		\end{proof}

	\begin{lem}\label{towards.uniqueness} Let $u,v\in X$ (defined in \eqref{set.X}), such that $u\ge \mathcal{T}(u)$ and $v\le \mathcal{T}(v)$. If $u>0$, then $u\ge v$.
	\end{lem}
	
	\begin{proof} 
		First, we show that $\inf_{x\in\Omega} u(x)>0.$ We will discuss this according to the value of $F(0).$  If $F(0)>0,$ then
		  we have \[\forall x\in\Omega,\quad u(x)\geq \mathcal{T}(u)(x)\ge2a\delta F(0)>0.\]  If $F(0)=0,$  then
		   \[\forall x\in\Omega,\quad u(x)\geq \mathcal{T}(u)(x)=\int_\Omega k(x,y)F(u(y))\,{\rm d}y \ge \delta\int_\Omega F(u(y))\,{\rm d}y>0,\]
		 Note that $\int_\Omega F(u(y))\,{\rm d}y>0,$ because we can find an interval in $\Omega$ where $u>0,$ which implies $F(u)>0$ on that interval. This proves  our claim. Also, we have \[v\le \mathcal{T}(v)=\int_\Omega k(\cdot,y)F(v(y))\,dy\le 2a \Lambda M.\]
		Since $u,v\in X$, we denote by $\{B_i\}_{1\le i\le p}$ and $\{C_j\}_{~ 1\le j\le q}$ the subintervals of $\Omega$ such that  $u$ and $v$ are continuous on each $B_i$ and $C_j$ respectively. Let $A_{ij}=B_i\cap C_j$ and define $$\alpha_{ij}:=\inf\left\{\alpha>0,~\alpha u(x)\geq v(x)\text{ for all }x\in A_{ij}\right\}.$$
		The infimum exists as the set $\left\{\alpha>0,~\alpha u_{|A_{ij}}\ge v_{|A{ij}}\right\}$ is nonempty, $\inf_{x\in\Omega} u(x)>0$ and $v$ is bounded above. It follows from the continuity of $u$ and $v$ over $A_{ij}$ that $\alpha_{ij}u\ge v$ in $A_{ij}$. Moreover, from the definition of $\inf,$ there exists $x_{ij}\in A_{ij}$ such that $\alpha_{ij}u(x_{ij})=v(x_{ij}).$ Let \[\alpha_0=\max_{ij}(\alpha_{ij})=\alpha _{i_0j_0}.\]   Then, $\alpha_0u\ge v$  in $\Omega.$ Denote by $x_0:=x_{i_0j_0}$. We claim that $\alpha_0\le1,$ which would then prove our lemma. Suppose to the contrary that $\alpha_0>1$. Then,
		$$ \begin{array}{lll}
			0 &= &\alpha_0u(x_0)-v(x_0)\vspace{5 pt}\\
			&\ge &\alpha_0\mathcal{T}( u)(x_0)-\mathcal{T}(v)(x_0)\vspace{7pt}\\
			&=  &\int_\Omega k(x_0,y)\left[\alpha_0F(u(y))-F(v(y))\right]{\rm d}y.
		\end{array}$$
		From assumption \eqref{F.cond}, and  since $F$ is increasing, we get 
		\[\alpha_0F(u(y)) = \alpha_0u(y)\frac{F(u(y))}{u(y)}
		> \alpha_0u(y)\frac{F(\alpha_0u(y))}{\alpha_0u(y)} 
		\geq F(v(y))\text{  for all }y.\]
		Thus,
		$$\int_\Omega k(x_0,y)\left[\alpha_0F(u(y))-F(v(y))\right]{\rm d}y>0,$$ which is contradiction. Hence, $u\ge \alpha_0u\ge v\text{ in }\Omega.$ Therefore, $u\geq v$ in $\Omega$. 
	\end{proof}

	Lemma \ref{towards.uniqueness} leads to  the uniqueness of a non-zero stationary solution of \eqref{IDE}. 
	
	\begin{corollary}[Uniqueness of the stationary state]\label{cor} There exists at most one non-zero stationary solution of \eqref{IDE} in $X$.
	\end{corollary}
	\begin{proof}
		In fact, if there exists  two non-zero stationary solutions $u$ and $v$, we  get $u\ge \mathcal{T}(u)$ and $v\le \mathcal{T}(v).$  Lemma \ref{equilibrium.is.positive} implies that $u>0$ and Lemma \ref{towards.uniqueness}  yields  that $u\ge v$. Also, we have $v\ge \mathcal{T}(v)$ and $u\le T(u).$ From Lemma \ref{equilibrium.is.positive}, we then get that  $v>0.$ Lemma \ref{towards.uniqueness}  implies that $v\ge u.$ Therefore, $u\equiv v$.
	\end{proof}

The following lemma will be used in proving the convergence in a certain mode 	of a sequence $\{u_n\}_n$ satisfying \eqref{IDE}. The lemma is announced in terms two types of  equicontinuity of a sequence of functions. We will recall  the definitions of pointwise and uniform equicontinuity in what follows and then announce the lemma involving these notions. 

\begin{definition}[Pointwise equicontinuous]A sequence of functions $\{f_n\}_n$ is said to be pointwise equicontinuous on a set $J$ if 
\begin{equation}\label{point.equi}
\begin{array}{l}
\forall \varepsilon>0,~\forall z\in J,~ \exists~ \delta(z)>0,\text{ such that }
|f_n(x)-f_n(z)|<\varepsilon,\vspace{5 pt}\\\text{ whenever }x\in J ~\text{ and } ~|x-z|<\delta(z).
\end{array}
\end{equation}
\end{definition}

\begin{definition}[Uniformly equicontinuous] A sequence of functions $\{f_n\}_n$ is said to be uniformly equicontinuous on a set $J$ if 
\begin{equation}\label{uniform.equi}
\begin{array}{l}
\forall \varepsilon>0,~ \exists~ \delta>0,\text{ such that }
|f_n(x)-f_n(y)|<\varepsilon\vspace{5 pt}\\\text{ whenever }x,~y\in J \text{ and } |x-y|<\delta.
\end{array}
\end{equation}

\end{definition}	
	
	\begin{lem}\label{equicontinuity}
		Let $J$ be a compact set. If $\left\{f_n\right\}_n\subset C(J)$ is pointwise equicontinuous, then  $\left\{f_n\right\}_n$ is uniformly equicontinuous.
	\end{lem}
	\begin{proof} Let $\varepsilon>0$. For every $z\in J$ , there exists $\delta_z>0$ such that for all $x\in B_J(z,\delta_z)$, we have $|f_n(x)-f_n(z)|<{\varepsilon}/{2}$ for all $n\in \mathbb{N}$, where $B(z,\delta_z)$  is the open ball of center $z$ and radius $\delta_z$ and $B_J(z,\delta_z)=J\cap B(z,\delta_z)$. We have \[\displaystyle{\bigcup_{z\in J}B_J\left(z,\frac{\delta_z}{2}\right)=J}.\] By compactness of $J$, there exist $z_1,\ldots,z_m\in J$ such that \[J=\bigcup_{i=1}^mB_J\left(z_i,\frac{\delta_{z_i}}{2}\right).\] We take $\delta_0<\min_{1\le i\le m}\left(\frac{\delta_{z_i}}{2}\right).$
		Let $x,y\in J$ such that $|x-y|<\delta_0$. Then, $x\in B_J\left(z_j,\frac{\delta_{z_j}}{2}\right),$ for some $1\le j\le m$. Now, $$|y-z_j|\le |x-y|+|x-z_j|<\delta_0+\frac{\delta_{z_j}}{2}\le \frac{\delta_{z_j}}{2}+\frac{\delta_{z_j}}{2}=\delta_{z_j}.$$ Thus, $y\in B_J(z_j,\delta_{z_j}).$ It then follows that $$|f_n(x)-f_n(y)|\le|f_n(x)-f_n(z_j)|+|f_n(y)-f_n(z_j)|<\frac{\varepsilon}{2}+\frac{\varepsilon}{2}=\varepsilon,$$ for all $n\in \mathbb{N}$. Hence, the set   $\left\{f_n\right\}_n$ is uniformly equicontinuous.

	\end{proof}
	
	\begin{lem}\label{T(X)}
		If $u_0\in X$, then the sequence $\{u_n\}_n,$ obtained from \eqref{IDE} with an initial condition $u_0,$ satisfies $u_n\in X$ for all $n\in\mathbb{N}.$ Moreover, if $\{u_n\}_n$ converges pointwise to $w\in  L^2(\Omega)$, then  $w\in X$.
	\end{lem}
	
	\begin{proof} As $u_0\in X$, there exists $m$ subintervals $\{A_i\}_{1\le i\le m}$  in  $\Omega,$ such that   $u_0$ is continuous over each $A_i$.
		We claim that for every $n$, $u_n$ is continuous on $\Omega$  except possibly at the points $\{a_i\}_{1\leq i\leq n}$ (introduced in assumption \ref{discontinuous}) and at  the discontinuity points of $u_0,$ which are finitely many. Let $B_{ij}$ be the interval defined by \[B_{ij}:=\Omega_{i}\cap A_j.\] For $x\in B_{ij},$ we have
		\[u_1(x)  = \int_\Omega k(x,y)F(u_0(y))\,{\rm d}y=\sum_{i,j}\int_{B_{ij}}k(x,y)F(u_0(y))\,{\rm d}y.\]
		Since $y\mapsto k(x,y)F(u_0(y))$ is continuous on each $B_{ij},$ it follows  that $u_1$ is continuous at $x$ as a sum of a finite number of  continuous functions. Hence, $u_1$ is continuous on $\Omega$ except possibly at the points mentioned in the claim. Suppose now that the claim is true  for all $u_i$  such that   $i\leq n$. Then, 
		$$
		u_{n+1}(x)  = \int_\Omega k(x,y)F(u_n(y))\,{\rm d}y=\sum_{i,j}\int_{B_{ij}}k(x,y)F(u_n(y))\,{\rm d}y.$$
		As the function $y\mapsto k(x,y)F(u_n(y))$ is continuous on each $B_{ij},$ it then follows that  $u_{n+1}$ is continuous at $x$. This proves the first assertion in our Lemma. 
		
		We move now to the proof of the second assertion. Suppose that $\{u_n\}_n$ converges pointwise to $w\in L^2(\Omega).$  We claim that $w$  is continuous on $\Omega$ except possibly at the points $\{a_i\}_{1\leq i \leq n}$ and at the points of discontinuity of $u_0.$ Suppose that there exists $z\in B_{i_0j_0}$ such that $w$ is discontinuous at $z$, for some $i_0$ and $j_0.$ Since $B_{i_0j_0}$ is open, there exists $r>0$ such that \[D:=[z-r,z+r]\subset B_{i_0j_0}.\] We have $\{u_n\}_n$ are all continuous on $D.$ Let us show that  $\{u_n\}_n$ is pointwise equicontinuous. Let $x_0\in D$ and $\varepsilon>0.$ The function $(x,y)\mapsto k(x,y)$ is uniformly continuous on $D\times B_{ij},$ for all $(i,j).$ This is because the function $k$ can be extended to a  continuous and bounded function on each of the compact sets $\overline{D\times B_{ij}}.$   Thus,  there exist $h_{ij}>0$ such that, for all $(x,y)$ and $(x_1,y_1)$ in $D\times B_{ij},$ we have 
		\[|(x-x_1,y-y_1)|<h_{ij}\implies |k(x,y)-k(x_1,y_1)|<l:=\frac{\varepsilon}{2aM},\] where $M$ is an upper bound of $F.$ Let $h:=\min_{i,j}\{h_{ij}\}.$ Then, for all $i,j$ and for all $(x,y)\in D\times B_{ij},$ we have
		\[|x-x_0|=|(x,y)-(x_0,y)|<h\implies |k(x,y)-k(x_0,y)|<l. \] 
		Thus, for every $n$ and for every $x\in D,$ such that $|x-x_0|<h,$ we have
		\[\begin{array}{lll}
			|u_n(x)-u_n(x_0)|&<&M\int_{-a}^a|k(x,y)-k(x_0,y)|\,{\rm d}y\vspace{5 pt}\\
			&\leq &M\sum_{i,j}\int_{B_{ij}}|k(x,y)-k(x_0,y)|\,{\rm d}y\vspace{5 pt}\\
			&<&M\sum_{i,j}\frac{\varepsilon}{2aM}|B_{ij}|=\varepsilon, 
		\end{array}\]
		where $|B_{ij}|$ denotes the Lebesgue measure of $B_{ij}.$
		Hence,  $\{u_n\}_n$ is pointwise equicontinuous. From Lemma \ref{equicontinuity}, we conclude that $\{u_n\}_n$ is uniformly equicontinuous because $D$ is compact. Moreover, $\{u_n\}_n\subseteq C(D)$ converges to $w$ pointwise and it is equicontinuous. Therefore, $\{u_n\}_n$ converges uniformly to $w$; hence, $w$ is continuous over $D$.
		This contradicts the assumption that $w$ is discontinuous at $z\in D$. The claim then follows and  $w\in X$. 
	\end{proof}

	\begin{lem}\label{monotone}
		If $u_1\ge u_0$ (respectively $u_1\le u_0$), then $\{u_n\}_n$ is increasing (respectively decreasing).
	\end{lem}
	
	\begin{proof}Suppose  that, for all $i
		\leq n-1,$ we have  $u_{i+1}\ge u_i.$ Then,
		$$u_{n+1}(x)=\int_\Omega k(x,y)F(u_n(y))\,{\rm d}y\geq \int_\Omega k(x,y)F(u_{n-1}(y))\,{\rm d}y=u_n(x),$$
		because $F$ is increasing.
		The other case can be proven similarly. 
	\end{proof}

	\begin{lem}\label{conv.in.L2} If $\{u_n\}_n$ converges pointwise to $w$, then $\{u_n\}_n$ converges to $w$  in $L^2(\Omega).$
	\end{lem}
	
	\begin{proof}Since $w$ is the limit of $\{u_n\}_n,$ it is a stationary solution: indeed 
		$$\mathcal{T}(w)=\mathcal{T}\left(\lim_{n\to\infty}u_n\right)=\lim_{n\to\infty}\mathcal{T}(u_n)=\lim_{n\to\infty}u_{n+1}=w.$$
		In the above equalities, we can interchange the limit and the integral by Lebesgue dominated convergence theorem.  Thus, 
		\[\|u_n-w\|_2^2=\|u_n-\mathcal{T}(w)\|_2^2=\int_\Omega\left(\int_{-a}^a k(x,y)(F(u_{n-1}(y))-F(w(y)))\,{\rm d}y\right)^2\,{\rm d}x.\]
		We have the following first inequality:
		$$\begin{array}{l}
			\left(\int_{-a}^a k(x,y)(F(u_{n-1}(y))-F(w(y)))\,{\rm d}y\right)^2\vspace{5 pt}\\
			\le\left(\int_{-a}^a[\Lambda(M+M)]\,{\rm d}y\right)^2
			\le 16a^2M^2\Lambda^2\in L^1(\Omega).
		\end{array}$$
		Moreover, we have 
		\[k(x,y)\left[F(u_{n-1}(y))-F(w(y))\right]\le2\Lambda M\in L^1(\Omega).\]  Hence, using the Lebesgue dominated convergence theorem, we obtain 
		$$\lim_{n\to\infty}\|u_n-w\|_2^2=\int_\Omega\left(\int_{-a}^a\lim_{n\to\infty} k(x,y)(F(u_{n-1}(y))-F(w(y)))\,{\rm d}y\right)^2\,{\rm d}x=0.$$ The proof of Lemma \ref{conv.in.L2} is now complete. 
	\end{proof}

	We are now in position to prove the three main theorems announced in Section \ref{intro}.

	\begin{proof}[{\bf Proof of Theorem \ref{convergence_to_equil}}]   We know that $0$ is a stationary solution of $\mathcal{T}$. We claim that there is no positive stationary solution. Suppose that there exists a positive stationary solution $w$  of $\mathcal{T}$. It then follows, from Lemma \ref{equilibrium.is.positive}, that  $w>0$ on $\Omega$. Using assumption \eqref{F.cond}, we have
		\begin{equation}\label{ineq}F(w(y))=w(y)\frac{F(w(y))}{w(y)}<r_0w(y),   \text{ for all  }y\in \Omega,\end{equation}
		 where $r_0=F'(0).$
		Thus,  \begin{equation}\label{karim}\begin{array}{ll}w(x)&=\int_\Omega k(x,y)F(w(y))\,{\rm d}y\vspace{5 pt}\\
			&<r_0\int_\Omega k(x,y)w(y)\,{\rm d}y=\mathcal{T}_0(w)(x),
		\end{array}\end{equation} for all $x\in \Omega$. We will now prove that there exists $\mu\in(0,1),$ such that 
		\begin{equation}\label{ramadan}w\le (1-\mu)\mathcal{T}_0(w)\text{ in }\Omega.
		\end{equation} 
		Let $\alpha:=\sup_{x\in\Omega}\frac{w(x)}{\mathcal{T}_0(w)(x)}.$ Then there exists $x_0\in\overline{\Omega},$ such that
		\[\lim_{x\rightarrow x_0}\frac{w(x)}{\mathcal{T}_0(w)(x)}=\alpha.\] From \eqref{karim}, we get $0< \alpha\leq 1.$ Our goal is to prove that $0<\alpha<1.$ Suppose to the contrary that $\alpha=1.$
		We then have  \begin{equation}\label{equal}\lim_{x\rightarrow x_0}{w(x)}=\lim_{x\rightarrow x_0}\mathcal{T}_0(w)(x).\end{equation}
		 But \eqref{ineq} and Lebesgue dominated convergence theorem lead to
		\[\begin{array}{ll}
		\lim_{x\rightarrow x_0}{w(x)}&=\lim_{x\rightarrow x_0}\int_\Omega k(x,y)F(w(y))\,{\rm d}y\vspace{5 pt}\\
		&=\int_\Omega \lim_{x\rightarrow x_0}k(x,y)F(w(y))\,{\rm d}y\vspace{5 pt}\\
		&<r_0\int_\Omega \lim_{x\rightarrow x_0}k(x,y)w(y)\,{\rm d}y\vspace{5 pt}\\
		&=\lim_{x\rightarrow x_0}r_0\int_\Omega k(x,y)w(y)\,{\rm d}y\vspace{5 pt}\\
		&=\lim_{x\rightarrow x_0}\mathcal{T}_0(w)(x).
		\end{array}\]
		This contradicts \eqref{equal} and thus $\alpha<1.$ This proves the inequality \eqref{ramadan}. We then iterate to get 
		$$w\le(1-\mu)r_0\int_\Omega k(\cdot,y)w(y)\,{\rm d}y\le (1-\mu)^2\mathcal{T}_0^{\,2}(w)\le\text{\ldots}\le(1-\mu)^n\mathcal{T}_0^{\,n}(w),$$
		for all $n\in\mathbb{N}$. 
		
		Let us show  that $\inf_\Omega(\phi_0)>0,$ where $\phi_0$ is the principal eigenfunction of $\mathcal{T}_0$ associated with the principal eigenvalue $\lambda_0.$ In fact,  
		since $k(x,y)>\delta$ for all $(x,y),$ then, for all  $x\in\Omega$ we  have  $$\phi_0(x)=\frac{r_0}{\lambda_0}\int_\Omega k(x,y)\phi_0(y)\,{\rm d}y>\frac{r_0\delta}{\lambda_0}\int_\Omega \phi_0(y)\,{\rm d}y>0,$$ because $\phi_0\not\equiv 0$ and $\phi_0\geq 0.$ Hence,  $$\inf_{\Omega}\phi_0\geq \frac{r_0\delta}{\lambda_0}\int_\Omega \phi_0(y)\,{\rm d}y>0.$$ 
		We recall that $\phi_0>0$ in $\Omega$ and it is unique up to multiplication by a scalar.  Thus, we can assume that $w\le\phi_0$  over $\Omega$ since $\inf_\Omega\phi_0>0.$ 
		Therefore, $$\mathcal{T}_0(w)\le r_0\int_\Omega k(x,y)w(y)\,dy\le r_0\int_\Omega k(x,y)\phi_0(y)\,dy=\mathcal{T}_0(\phi_0).$$ Suppose  that, for all $i
		\leq n-1,$ we have $\mathcal{T}^{\,i}_0(w)\le\mathcal{T}^{\,i}_0(\phi_0).$ Then, $$\mathcal{T}^{\,n}_0(w)\le r_0\int_\Omega k(x,y)\mathcal{T}^{\,n-1}_0(w)(y)\,dy\le r_0\int_\Omega k(x,y)\mathcal{T}^{\,n-1}_0(\phi_0)(y)\,dy\le\mathcal{T}^{\,n}_0(\phi_0).$$ By induction, and using the assumption that $\lambda_0\leq 1,$ we then obtain that  $$ w\le(1-\mu)^n\mathcal{T}_0^{\,n}(w)\le (1-\mu)^n\mathcal{T}_0^{\,n}(\phi_0)=(1-\mu)^n\lambda_0^{n}\phi_0\leq (1-\mu)^n\phi_0,$$
	for all $n\in\mathbb{N}.$
		Since $0<(1-\mu)<1,$ passing to the limit as $n\rightarrow+\infty$ in the above inequality yields   $w\equiv0$ in $\Omega,$  which is contradiction. This completes the proof of the first result in the theorem.

		We move now to the proof of the second result in this theorem. Let $u_0\in X$ and fix $N>2a\Lambda M.$ Then, \[\mathcal{T}(N)=\int_{-a}^ak(x,y)F(N)\,{\rm d}y\le2a\Lambda M<N.\]
		From Lemma \ref{monotone}, it then follows that  $\{\mathcal{T}^{\,n}(N)\}_n$ is decreasing. Since $\{\mathcal{T}^{\,n}(N)\}$ is bounded below by $0$ uniformly, $\{\mathcal{T}^{\,n}(N)\}_n$ converges pointwise to $w.$  Lemma \ref{T(X)} yields the  existence of $w\in  X.$ From  the Lebesgue dominated convergence theorem, it follows  that 
		$$\mathcal{T}(w)=\mathcal{T}\left(\lim_{n\to\infty}\mathcal{T}^{\,n}(N)\right)=\lim_{n\to\infty}\mathcal{T}^{\,n+1}(N)=w.$$ The above means that 
		$w$ is stationary solution of $\mathcal{T}$. Since $\lambda_0\le1,$ it follows, from the first part of  Theorem 1,  that $w=0$. Now,   \[0\leq u_1=\mathcal{T}(u_0)\le2r\Lambda M<N.\]
		Suppose that, for  all $i\leq n,$ we have $u_i\le \mathcal{T}^{\,i-1}(N).$ Then,
		\[u_{n+1}=\int_{-a}^ak(\cdot,y)F(u_n(y))\,{\rm d}y\le\int_{-a}^ak(\cdot,y)F(\mathcal{T}^{\,n-1}(N)(y))\,{\rm d}y=\mathcal{T}^{\,n}(N).\]
		By induction, it then follows that  \[0\le u_n\le \mathcal{T}^{\,n-1}(N) \text{ for all }n\in\mathbb{N}.\] Taking the limits in the last inequality yields  that $\{u_n\}_n$ converges to $0$ pointwise. From Lemma \ref{conv.in.L2}, it follows that $\{u_n\}_n$ converges to $0$ in $L^2(\Omega)$.
	\end{proof}
	\begin{proof}[{\bf Proof of Corollary \ref{r0Theorem}}]
    From Theorem \ref{convergence_to_equil}, it suffices to prove that $\lambda_0\le1.$ First, we claim that $\varphi\in L^\infty(\Omega).$ We have 
    \begin{equation}
    \begin{array}{l}
        0<\lambda_0\varphi(x)=\int_\Omega k(x,y)\varphi(y)\,\mathrm{d}y\le \|k(x,\cdot)\|_{L^2(\Omega)}\|\varphi\|_{L^2(\Omega)}\\
        \le 2a\Lambda \|\varphi\|_{L^2(\Omega)}<\infty,~~ \text{ as $\varphi\in L^2(\Omega).$}
        \end{array}
    \end{equation}
     Hence, our claim that $\varphi\in L^\infty(\Omega)$ follows. Choosing $\varphi $ such that $\|\varphi\|_{L^\infty(\Omega)}=1,$ we then obtain  
    \[\begin{array}{cl}
    \lambda_0=\|\lambda_0\varphi\|_{L^\infty(\Omega)}&=\|\mathcal{T}_0(\varphi)\|_{L^\infty(\Omega)}\vspace{7 pt}\\
    &=\sup_\Omega \left|r_0\int_\Omega k(x,y)\varphi(y)\,\mathrm{d}y\right|\le \sup_\Omega r_0\int_\Omega k(x,y)\|\varphi\|_{L^\infty(\Omega)}\,\mathrm{d}y\vspace{7 pt}\\
    &=\sup_\Omega r_0\int_\Omega k(x,y)\,\mathrm{d}y
    \le\sup_\Omega r_0=r_0\le1.
    \end{array}\]
 This completes the proof of Corollary \ref{r0Theorem}.
\end{proof}		
\noindent We are now in position to prove Theorem \ref{Nonzero_Equilibrium}.
	
	\begin{proof}[{\bf Proof of Theorem \ref{Nonzero_Equilibrium}}]  Here, we have the assumption that $\lambda_0>1$ and $F(0)=0.$ Appealing to Lemma \ref{equilibrium.is.positive} and Corollary \ref{cor}, it suffices to find one positive stationary solution of $\mathcal{T}.$

		Let $h>0$ such that $\frac{r_0}{1+h}>1$ and $\frac{\lambda_0}{1+h}>1.$ Note that $h$ exists because $r_0>1$ and $\lambda_0>1.$ Since \[
		\begin{array}{ll}
			\lambda_0\phi_0(x)=\mathcal{T}_0(\phi_0)(x)&=r_0\int_{-a}^ak(x,y)\phi_0(y)\,{\rm d}y\vspace{5 pt}\\
			&\leq r_0\left(\int_{-a}^a(k(x,y))^2\,{\rm d}y\right)^{1/2}\left(\int_{-a}^a(\phi_0(y))^2\,{\rm d}y\right)^{1/2}
			\vspace{5 pt}\\
			&\le r_0\sqrt{2a}\Lambda\|\phi_0\|_2,
		\end{array}\] it follows that $\phi_0$ is bounded above by some constant $\kappa$ for all $x$.

		\noindent We claim that the limit \[\lim_{\varepsilon\to0}\frac{F(\varepsilon\phi_0(x))}{\varepsilon\phi_0(x)}=F^\prime(0)\] is uniform in $x\in\Omega.$ 
		
		From the definition of $r_0=F^\prime(0),$ we have that for any $\zeta>0$, there exists $\delta_0>0,$ such that  
		\[\text{for }  |t|<\delta_0,~\left|\frac{F(t)-F(0)}{t}-r_0\right|<\zeta.\]
		Let $\varepsilon^*>0$  be such that $\varepsilon^*\phi_0(x)<\varepsilon^*\kappa<\delta_0.$ In other words, $\varepsilon^*$ is small enough so that $\delta_0$  is a uniform upper bound of $\varepsilon^*\phi_0$. Hence, the claim follows. Thus, for $\zeta=\frac{hr_0}{1+h},$  we can find $\varepsilon_0>0,$ such that for all $\varepsilon\in(0,\varepsilon_0]$ we have $$\left|\frac{F(\varepsilon\phi_0(x))}{\varepsilon\phi_0(x)}-r_0\right|<\frac{hr_0}{1+h},~ \text{ for all }x\in\Omega.$$ Thus, \[\frac{F(\varepsilon\phi_0(x))}{\varepsilon\phi_0(x)}>\frac{r_0}{1+h},\] for all $\varepsilon\in(0,\varepsilon_0].$ Then, for all $\varepsilon\leq \varepsilon_0,$
		\begin{equation}\label{alpha}
			\begin{array}{ll}
				\mathcal{T}(\varepsilon \phi_0)&=\int_{\Omega} k(\cdot,y)\varepsilon\phi_0(y)\frac{F(\varepsilon\phi_0(y)}{\varepsilon\phi_0(y)}\,{\rm d}y\vspace{5 pt}\\
				&\geq\frac{r_0}{1+h}\int_\Omega k(\cdot,y)\varepsilon\phi_0(y)\,{\rm d}y=\frac{1}{1+h}\mathcal{T}_0(\varepsilon\phi_0)\vspace{5 pt}\\
				&=\frac{\lambda_0}{1+h}\varepsilon\phi_0\vspace{5 pt}\\
				&>\varepsilon\phi_0.
		\end{array}\end{equation}
		We claim that $\inf_\Omega(\mathcal{T}(\varepsilon\phi_0)-\varepsilon\phi_0)>0$. Suppose that $\inf_\Omega(\mathcal{T}(\varepsilon\phi_0)-\varepsilon\phi_0)=0.$ Then, there exists $x_0\in\bar{\Omega}$ such that \[\lim_{x\to x_0}\mathcal{T}(\varepsilon\phi_0)(x)=\lim_{x\to x_0}\varepsilon\phi_0(x)=\alpha.\]  From  \eqref{alpha}, we obtain 
		\[\alpha=\lim_{x\rightarrow x_0}\mathcal{T}(\varepsilon\phi_0)(x)\geq\lim_{x\to x_0}\frac{\lambda_0}{1+h}\varepsilon\phi_0(x)\geq \lim_{x\to x_0}\varepsilon\phi_0(x)=\alpha. \]
		This implies that $$\frac{\lambda_0}{1+h}\alpha=\alpha;$$
		 hence,  $\alpha=0$. But, 
		$$\begin{array}{lll}
			0&=\alpha&=\lim_{x\to x_0}\mathcal{T}(\varepsilon\phi_0)(x)\\
			&&=\lim_{x\to x_0}\int_\Omega k(x,y)F(\varepsilon\phi_0(y))\,dy\ge\delta\int_\Omega F(\varepsilon\phi_0(y))\,dy>0,
		\end{array}$$ and so we have a contradiction. Therefore, $ \inf_\Omega(\mathcal{T}(\varepsilon\phi_0)-\varepsilon\phi_0)>0.$ 
		
		Since the set of simple functions is dense in $L^\infty(\Omega),$ and $\varepsilon\phi_0\in L^\infty(\Omega)$,  there exists a sequence of simple functions $\{f_n\}_n$ that decreases uniformly to $\varepsilon\phi_0$. Thus, for \[\zeta_0:=\inf_\Omega(\mathcal{T}(\varepsilon\phi_0)-\varepsilon\phi_0)>0,\] there exists $n_\varepsilon\in\mathbb{N}$ such that, for all $n\ge n_\varepsilon,$ we have $|f_n(x)-\varepsilon\phi_0(x)|<\zeta_0$ for all $x\in \Omega$. Then, as $\mathcal{T}$ is increasing (which follows from the fact that $F$ is increasing), we get 
		$$\mathcal{T}(f_n)\ge \mathcal{T}(\varepsilon\phi_0)\ge \varepsilon\phi_0+\zeta_0>f_n$$ on $\Omega,$ for all $n\ge n_\varepsilon$. Since $f_n$ is a simple function, $f_n\in X.$  Now, for $n\ge n_\varepsilon$ we consider   the sequence $\{\mathcal{T}^{\,m}(f_n)\}_m.$  Since  $\mathcal{T}(f_{n})\ge \mathcal{T}(\varepsilon\phi_0) \ge f_{n}$ on $\Omega,$ 
		 Lemma \ref{monotone} implies that $\{\mathcal{T}^{\,m}(f_n)\}_m$ is increasing for all $\varepsilon\leq \varepsilon_0$ and for all $n\ge n_\varepsilon.$ Thus, for $0<\varepsilon<\varepsilon_0,$ we have $$\mathcal{T}^{\,m}(f_n)=\int_{-a}^ak(\cdot,y)F(\mathcal{T}^{\,m-1}(f_n)(y))\,{\rm d}y\le 2a\Lambda M,$$
		where $\Lambda$ and $M$ are the bounds of $k$ and $F$ respectively. Hence, $\{\mathcal{T}^{\,m}(f_n)\}_m$ is uniformly bounded and increasing. Therefore,  there exists $w\in X$ (by Lemma \ref{T(X)}) such that $w=\lim_{m\to\infty}\mathcal{T}^{\,m}(f_n)$ pointwise. The function $w$ is a positive stationary solution of $\mathcal{T}$   since 
		$$\mathcal{T}(w)=\mathcal{T}\left(\lim_{m\to\infty}\mathcal{T}^{\,m}(f_n)\right)=\lim_{m\to\infty}\mathcal{T}(\mathcal{T}^{\,m}(f_n))=\lim_{m\to\infty}\mathcal{T}^{\,m+1}(f_n)=w.$$
		Note that we were able to interchange the limit and the integral in the above equation because of the  Lebesgue dominated convergence theorem.

		\noindent We have
		\begin{equation}\label{bound}0<\delta\int_\Omega F(u_0(y)){\rm d}y\leq \mathcal{T}(u_0)\leq u_1\leq 2a\Lambda M.\end{equation}We fix $N>2a\Lambda M.$ Then, $\mathcal{T}(N)\le2a\Lambda M<N$. Appealing to Lemma \ref{monotone}, we get  $\{\mathcal{T}^{\,n}(N)\}_n$ is decreasing  and it is bounded  uniformly by $0.$     From Lemma \ref{T(X)}, there exists $w^*\in X$ such that $\{\mathcal{T}^{\,n}(N)\}_n$ converges pointwise to $w^*.$ This means that $w^*$ is a stationary solution of $\mathcal{T}$. Then, for $\varepsilon\leq \varepsilon_0,$ we have \[\mathcal{T}(f_n)<N,\text{ for all }n\ge n_\varepsilon.\] We know that $F$ is increasing. By induction, it  follows that  \[\mathcal{T}^{\,n}(f_n)\le \mathcal{T}^{\,n}(N).\]  Taking the limit on both sides,   we get that $w\le w^*$. From the uniqueness of the non-zero stationary solution of $\mathcal{T}$ in $X$, which was established in Corollary \ref{cor}, we have $w^*=w\not\equiv0$. From \eqref{bound}, we can find $\varepsilon\in(0,\varepsilon_0]$ such that \[\varepsilon\phi_0\le\varepsilon\phi_0+\nu\le u_1\le N\text{ in }\Omega,\] for some $\nu>0.$ Then, there exists $n_0$ such that, for all $n\ge n_0,$ we have $f_n< \varepsilon\phi_0+\nu$ in $\Omega.$  Let $n_1>\max\{n_{0},n_\varepsilon\}$. Then, $f_{n_1}\le u_1\le N.$
		Suppose that, for all $i\le n-1$, we have $$\mathcal{T}^{\,i}(f_{n_1})\le u_{i+1}\le \mathcal{T}^{\,i}(N).$$ Then, since $F$ is increasing, we obtain
		$$\begin{array}{ll}\mathcal{T}^{\,n}(f_{n_1})&=\int_\Omega k(\cdot,y)F(\mathcal{T}^{\,n-1}(f_{n_1}(y)))\,{\rm d}y\vspace{5 pt}\\
			&\leq\int_\Omega k(\cdot,y)F(u_n(y))\,{\rm d}y\vspace{5 pt}\\
			&\leq \int_\Omega k(\cdot,y)F(\mathcal{T}^{\,n-1}(N)(y))\,{\rm d}y.
		\end{array}$$
		Hence, $$\mathcal{T}^{\,n}(f_{n_1})\le \ u_{n+1}\le \mathcal{T}^{\,n}(N),\text{ for all }n\in \mathbb{N}.$$
		By passing to the limit as $n\rightarrow+\infty,$ we get $$w\le \lim_{n\to\infty}u_{n+1}\le w.$$
		Therefore, $\{u_n\}_n$ converges pointwise to $w,$ which is the unique positive stationary solution of $T$ over $\Omega.$ From Lemma \ref{conv.in.L2},  we obtain that $\{u_n\}_n$ converges to $w$  in $L^2(\Omega)$. This completes the proof of Theorem \ref{Nonzero_Equilibrium}.
	\end{proof}

	\medskip

	\begin{proof}[{\bf Proof of Theorem \ref{case_F(0)_>0}}]
		In this proof, we use the assumption  that $F(0)>0$. We fix $N>2a\Lambda M$ and 
 note that	
		\[\mathcal{T}(N)=\int_{-a}^ak(x,y)F(N)\,{\rm d}y\le2a\Lambda M<N.\]  
		 Lemma \ref{monotone} then yields that $\{\mathcal{T}^{\,n}(N)\}$ is decreasing. Since $\{\mathcal{T}^{\,n}(N)\}$  is bounded below by $0$ uniformly,  Lemma \ref{T(X)} leads to the existence of  $w\in  X,$  such that $\{\mathcal{T}^{\,n}(N)\}_n$ converges pointwise to $w.$ That is, $w$ is stationary solution of $\mathcal{T}.$ From Corollary \ref{cor}, it follows that $w$ is the unique stationary solution of \eqref{IDE}. We have $$\mathcal{T}(0)(x)=\int_\Omega k(x,y)F(0)\,dy\ge2a\delta F(0)>0 ~\text{ for all }x\in\Omega.$$ From Lemma \ref{monotone}, it follows that    $\{\mathcal{T}^{\,n}(0)\}_n$ is increasing. Since $\{\mathcal{T}^{\,n}(0)\}_n$ is uniformly bounded above by $N,$ Lemma \ref{T(X)} then yields the existence of $w_0\in X,$ such that $\{\mathcal{T}^{\,n}(N)\}_n$ converges pointwise to $w_0.$ This means that $w_0$ is stationary solution of $\mathcal{T}.$ Corollary \ref{cor} yields that $w_0=w,$  the unique stationary solution of \eqref{IDE}. Also, we have \[N\ge T(u_0)\ge 0.\] Suppose that it holds true that,  for all  $i\leq n$,  \[\mathcal{T}^{\,n-1}(N)\ge u_n\ge \mathcal{T}^{\,n-1}(0).\]
		We compute		
		$$\begin{array}{ll}\mathcal{T}^{\,n}(N)(x)&=\int_\Omega k(x,y)F(\mathcal{T}^{\,n-1}(N)(y))\,dy\vspace{5 pt}\\
			&\geq\int_\Omega k(x,y)F(u_n(y))\,dy\vspace{5 pt}\\
			&=u_{n+1}(x)\vspace{5 pt}\\&
			\geq \int_\Omega k(x,y)F(\mathcal{T}^{\,n-1}(0)(y))\,dy=\mathcal{T}^{\,n}(0)(x).
		\end{array}
		$$
		Thus, \[\mathcal{T}^{\,n-1}(N)\ge u_n\ge \mathcal{T}^{\,n-1}(0),\] for all $n\in \mathbb{N}.$ Taking the limits in the above inequality, we obtain that $$w\ge\lim_{n\to\infty}u_n\ge w.$$
		Hence $\{u_n\}_n$ converges to $w$ pointwise. From Lemma \ref{conv.in.L2}, it follows that $\{u_n\}_n$ converges to $w$ in $L^2(\Omega).$
		\end{proof}

\end{document}